\documentclass[12pt,twoside,reqno]{amsart}
\usepackage[all]{xy}   
\usepackage {amssymb,latexsym,amsthm,amsmath}
\usepackage{hyperref}
\usepackage{graphicx}
\usepackage{xcolor}
\usepackage[utf8]{inputenc}
\usepackage{palatino}
\usepackage{todonotes}
\usepackage{enumitem}
\usepackage{xpatch}

\long\def\symbolfootnote[#1]#2{\begingroup%
\def\thefootnote{\fnsymbol{footnote}}\footnote[#1]{#2}\endgroup}

\newcommand{\Q}{\mathbb Q}
\newcommand{\R}{\mathbb R}
\newcommand{\C}{\mathbb C}

\newcommand{\Aut}{\textup{Aut}}
\newcommand{\Inn}{\textup{Inn}}
\newcommand{\Out}{\textup{Out}}

\newcommand{\diag}{\textup{diag}}

\newcommand{\secref}[1]{Section~\ref{#1}}
\newcommand{\thmref}[1]{Theorem~\ref{#1}}
\newcommand{\lemref}[1]{Lemma~\ref{#1}}

\newcommand{\corref}[1]{Corollary~\ref{#1}}

\newcommand{\defref}[1]{Definition~(\ref{#1})}
\makeatletter
\def\imod#1{\allowbreak\mkern10mu({\operator@font mod}\,\,#1)}
\makeatother

\newtheorem{theorem}{Theorem}[section]
\newtheorem{lemma}[theorem]{Lemma}
\newtheorem{corollary}[theorem]{Corollary}

\newtheorem*{theorem*}{Theorem}
\theoremstyle{definition}
\newtheorem{remark}[theorem]{Remark}
\newtheorem{notation}[theorem]{Notation}
\newtheorem{question}[theorem]{Question}
\newtheorem{example}[theorem]{Example}
\newtheorem{definition}[theorem]{Definition}

\numberwithin{equation}{section}

\newcommand{\ignore}[1]{}

\newcommand{\mynote}[1]{}
\def \GL {\mathrm{GL}}
\def \SL {\mathrm{SL}}

\begin{document}
\setcounter{section}{0}
\setcounter{tocdepth}{1}
\title[Twisted Conjugacy Classes in Twisted Chevalley Groups]{Twisted Conjugacy Classes in Twisted Chevalley Groups}
\author[Sushil Bhunia]{Sushil Bhunia}
\author[Pinka Dey]{Pinka Dey}
\author[Amit Roy]{Amit Roy}
\thanks{Dey and Roy acknowledge financial support from UGC and CSIR, Govt. of India, respectively}
\address{Indian Institute of Science Education and Research (IISER) Mohali, Knowledge City,  Sector 81, S.A.S. Nagar 140306, Punjab, India}
\email{sushilbhunia@gmail.com}
\email{pinkadey11@gmail.com}
\email{amitiisermohali493@gmail.com, amitroy@iisermohali.ac.in}
\date{\today}
\begin{abstract}
Let $G$ be a group and $\varphi$ be an automorphism of $G$. Two elements $x, y\in G$ are said to be $\varphi$-twisted conjugate if $y=gx\varphi(g)^{-1}$ for some $g\in G$. We say that a group $G$ has the $R_{\infty}$-property if the number of $\varphi$-twisted conjugacy classes is infinite for every automorphism $\varphi$ of $G$. In this paper, we prove that twisted Chevalley groups over a field $k$ of characteristic zero have the $R_{\infty}$-property as well as the  $S_{\infty}$-property if $k$ has finite transcendence degree over $\Q$ or $\Aut(k)$ is periodic. 
\end{abstract}
\maketitle
\section{Introduction}
Let $G$ be a group and $\varphi$ be an automorphism of $G$. Two elements $x, y$ of $G$ are said to be twisted $\varphi$-conjugate, denoted by $x\sim_{\varphi} y$, if $y=gx\varphi(g)^{-1}$ for some $g\in G$. Clearly, $\sim_{\varphi}$ is an equivalence relation on $G$. The equivalence classes with respect to this relation are called the \textit{$\varphi$-twisted conjugacy classes} or \textit{the Reidemeister classes} of $\varphi$. If $\varphi=\mathrm{Id}$, then the $\varphi$-twisted conjugacy classes are  the usual conjugacy classes. The $\varphi$-twisted conjugacy class containing $x\in G$ is denoted by $[x]_{\varphi}$. The \textit{Reidemeister number} of $\varphi$, denoted by $R(\varphi)$, is the number of all $\varphi$-twisted conjugacy classes.  
A group $G$ has the \textit{$R_{\infty}$-property} if $R(\varphi)$ is infinite for every automorphism $\varphi$ of $G$.

The Reidemeister number is  closely related to the Nielsen number of a selfmap of a manifold, which is homotopy invariant. More precisely, for a compact connected manifold $ M $ of dimension at least $ 3 $ and a homeomorphism $f: M\to M$, the minimal number of the fixed-point set varying over all the maps  homotopic to $ f $ is same as the Nielsen number $ N(f) $. This number is bounded above by the Reidemeister number $R(f_{\#})$, where $f_{\#} $ is the induced map on $ \pi_1(M) $. If all the essential fixed-point classes have the same fixed-point index, then either $ N(f)=0 $ or $ N(f)=R(f_{\#}) $ (we refer the reader to \cite[p. 37, Theorem 5.6]{Jiang} for details). Since $ N(f) $ is always finite, we get $ N(f)=0 $ when $ R(f_{\#}) $ is infinite. Thus,  
$ R(f_{\#}) $ is infinite implies that we can deform $f$ to a fixed-point free map, up to homotopy.

 The problem of determining which classes of groups have the $R_{\infty}$-property is an active area of research initiated by Fel'shtyn and Hill \cite{fh94}. 
This problem has a long list of references, for example, see \cite{fg08, levitt07, MS, timur12,  timurnasy2016} just to name a few. 
For a nice introduction to the subject, we refer the readers to \cite{fel10, FN} and the references therein.
In this context, we recall a classical result by Steinberg, which says that for a connected linear algebraic group $G$ over an algebraically closed field $k$ and a surjective endomorphism $\varphi$ of $G$, if $\varphi$ has a finite set of fixed points then $G=\{g\varphi(g)^{-1}\mid g\in G\}=[e]_{\varphi}$, i.e., $R(\varphi)=1$ (see \cite[Theorem 10.1]{St}).  
In particular, any semisimple linear algebraic group $G$ over an algebraically closed field $k$ with $\operatorname{char}k=p>0$ possesses an automorphism $\varphi$ (Frobenius automorphism) with finitely many fixed points. Then by the Steinberg theorem, it implies that $R(\varphi)=1$. 

The theory of Chevalley groups was introduced by Claude Chevalley himself in \cite{chevalley55}, and further developed by Steinberg \cite{steinberg59}. Motivation of the present work comes from the results of Nasybullov and Fel'shtyn. They proved that a Chevalley group $G$ over a field $k$ of  characteristic zero possesses the  $R_{\infty}$-property if the transcendence degree of $k$ over $\mathbb{Q}$ is finite and $G$ is of type $\Phi$ (here $\Phi$ is a root system corresponding to $G$), or $\Aut(k)$ is periodic and $G$ is of type $\Phi\neq A_1$ (see \cite[Theorem 3.2]{FN}. Also, see \cite[Theorems 1, 2]{NaT} for details). 
It is worth mentioning that a reductive linear algebraic group $G$ over an algebraically closed field $k$ of characteristic zero possesses the  $R_{\infty}$-property if the transcendence degree of $k$ over $\Q$ is finite and the radical of $G$ is a proper subgroup of $G$ (see \cite[Theorem 4.1]{FN}). Later Nasybullov generalized the results from Chevalley groups over fields to Chevalley groups (of classical types) over certain rings. A Chevalley group $G$ of type $A_l, B_l, C_l$ or $D_l$ over a local integral domain $k$ of zero characteristic possesses the $R_{\infty}$-property if $\Aut(k)$ is periodic (see \cite[Theorems 1, 2]{timur12} and \cite[Theorem 1]{timurnasy2016}). 
Recently, Nasybullov showed that Chevalley groups (of types $A_n, B_n, C_n, D_n$) over an algebraically closed field $k$ do not satisfy the $R_{\infty}$-property, when $k$ has infinite transcendence degree over $\Q$ (see \cite[Theorem 7]{timur2019} and \cite[Theorem 8]{nasy2019}).  So unless otherwise specified, we will assume that the field has zero characteristic. 
\vskip 2mm

The following natural question arises in this context: 
\begin{question}\label{question1}
 What can we say about the $R_{\infty}$-property of the twisted Chevalley groups (also known as Steinberg groups)?
\end{question}
The main results of this paper are the following, which solves Question \ref{question1} over a field $k$ of characteristic zero with finite transcendence degree over $\Q$ or $\Aut(k)$ is periodic. In a sequel,
we will do this for twisted Chevalley groups over an integral domain $k$ under the same assumptions on $k$ as Nasybullov did. Since we are considering only zero characteristic field, then non-trivial graph automorphisms exist only for the root systems of types 
$A_l \, (l\geq 2)$, $D_l\, (l\geq 4)$ and $E_6$. Suppose $\Phi$ is an irreducible root system corresponding to the adjoint Chevalley group $G$ of types either $A_l \, (l\geq 2)$ or $D_l\, (l\geq 4)$ or $E_6$. Let $\sigma=\overline{\rho}f$ be a product of graph and field automorphisms of $G$, and $G'_{\sigma}$ be the corresponding twisted Chevalley group of adjoint type (which is simple). In general any twisted Chevalley group also denoted by $G'_{\sigma}:=\widetilde{G}'_{\sigma}/Z$, where $\widetilde{G}'_{\sigma}$ is the twisted Chevalley group of universal type, and $Z\leq Z(\widetilde{G}'_{\sigma})$ is a central subgroup (for details, see \secref{tchev}). Now that we have the notations laid out, we state our main theorems of this paper. 
\begin{theorem}\label{mainthm1}
Let $G'_{\sigma}$ be a twisted Chevalley group corresponding to the irreducible root system $\Phi$ over a field $k$ of $\mathrm{char}\; k=0$. If the transcendence degree of $k$ over $\mathbb{Q}$ is finite, then $G'_{\sigma}$ possesses the $R_{\infty}$-property.
\end{theorem}
Examples of such fields are $\Q(T_1, T_2, \ldots, T_n), \overline{\Q}, \overline{\Q}(T_1, T_2, \ldots, T_n)$, etc.
\begin{theorem}\label{mainthm2}
Let $G'_{\sigma}$ be a twisted Chevalley group corresponding to the irreducible root system $\Phi$ over a field $k$ of $\mathrm{char}\; k=0$. If $\Aut(k)$ is periodic, then $G'_{\sigma}$ possesses the $R_{\infty}$-property.
\end{theorem}
Some examples of such fields are: any number field, the real numbers $\mathbb{R}$, the $p$-adic fields $\Q_p$, etc. 

Further Theorems \ref{mainthm1} and \ref{mainthm2} provide a characterization of the $S_{\infty}$-property in  twisted Chevalley groups. 
Suppose $\Psi\in \Out(G):=\Aut(G)/\Inn(G)$. Two elements $\alpha, \beta\in \Psi$ are said to be \emph{isogredient} (or similar) if $\beta=i_g\circ \alpha \circ i_{g^{-1}}$ for some $g\in G$, where $i_g(h)=ghg^{-1}$. Clearly, this is an equivalence relation on $\Psi$. Let $S(\Psi)$ denote the number of isogredience classes of $\Psi$. A group $G$ has the \emph{$S_{\infty}$-property} if $S(\Psi)=\infty$ for all $\Psi\in \Out(G)$. (For details see \secref{isogred}. Also, see \cite{ft2015}). Any non-elementary hyperbolic group satisfies the $S_{\infty}$-property (see \cite{ll00}). It follows from the works of Nasybullov that the Chevalley group has the $S_{\infty}$-property under the same condition as above on the field $k$ (although it was implicit there). For example, see \cite{FN}. In this direction, we have the following result: 
\begin{corollary}\label{maincor1}
Let $G'_{\sigma}$ be a twisted Chevalley group corresponding to the irreducible root system $\Phi$ over a field $k$ of $\mathrm{char}\; k=0$. If the transcendence degree of $k$ over $\mathbb{Q}$ is finite or $\Aut(k)$ is periodic, then $G'_{\sigma}$ possesses the $S_{\infty}$-property.
\end{corollary}
The problem of determining when the identity class $[e]_{\varphi}$ is a subgroup of $G$ was initiated by Bardakov et al. in \cite{bnn2013}. Also, see \cite{gn2019} for some recent works. In \cite[Theorems 3, 4]{NaT}, Nasybullov proved that $[e]_{\varphi}$ is a subgroup of a Chevalley group $G$ over $k$ if and only if $\varphi$ is a central automorphism, provided that the field $k$ over $\Q$ has finite transcendence degree or $\Aut(k)$ is periodic. We extend this result to the twisted Chevalley groups and prove the following: 
\begin{corollary}\label{maincor2}
Let $G'_{\sigma}$ be a twisted Chevalley group corresponding to the irreducible root system $\Phi$ over a field $k$ of $\mathrm{char}\; k=0$. Suppose that the transcendence degree of $k$ over $\mathbb{Q}$ is finite or $\Aut(k)$ is periodic. 
Then the $\varphi$-twisted conjugacy class $[e]_{\varphi}$ is a subgroup of $G'_{\sigma}$ if and only if $\varphi$ is a central automorphism of $G'_{\sigma}$.
\end{corollary}

\subsubsection*{Structure of the paper}
This paper is the study of the $R_{\infty}$ and $S_{\infty}$-property of the twisted Chevalley groups. 
In \secref{preliminaries}, we cover the preliminaries.
 \secref{mainsection} contains the proof of \thmref{mainthm1} and \thmref{mainthm2}. The final section is devoted to the proof of \corref{maincor1} and \corref{maincor2}. 

\section{Preliminaries}\label{preliminaries}
In this section, we fix some notations and terminologies, and recall some results which will be used throughout this paper. Most of the notions are from Carter \cite{ca}. Let $k$ be a field of characteristic zero.
\subsection{Chevalley groups} \label{section21}
We refer the interested reader to the original work of Chevalley \cite{chevalley55}. Let $\mathcal{L}$ be a complex simple Lie algebra and $\mathcal{H}$ be a Cartan subalgebra of $\mathcal{L}$. Consider the adjoint representation  
$\mathrm{ad} : \mathcal{L} \rightarrow \mathfrak{gl}(\mathcal{L})$
given by $\mathrm{ad}X(Y)=[X,Y]$.
Thus we have the Cartan decomposition 
	$$\mathcal{L}=\mathcal{H}\bigoplus \displaystyle\sum_{\alpha \in \Phi}\mathcal{L}_{\alpha},$$ 
	where $\mathcal{L}_{\alpha}=\{X\in \mathcal{L}\mid \mathrm{ad}H(X)=\alpha(H)X,\text{ for all } H \in \mathcal{H}\}$ 
	are root spaces of $\mathcal{L}$ and $\Phi$ is an irreducible root system with respect to $\mathcal{H}$. Also, we fix $\Delta$ and $\Phi^+$ to be a simple (or fundamental) root system and a positive root system, respectively. Therefore $\dim_\C(\mathcal{L})=|\Phi|+|\Delta|$. 
Chevalley proved that there exists a basis of $\mathcal{L}$ such that all the structure constants, which define $\mathcal{L}$ as a Lie algebra, are integers (for example, see \cite[Theorem 4.2.1]{ca}). This is a key theorem to define Chevalley groups. Let $h_{\alpha} \in \mathcal{H}$ 
be the co-root corresponding to the root $\alpha$. Then, for each root $\alpha \in \Phi$,
	an element $e_{\alpha}$ can be chosen in $\mathcal{L}_{\alpha}$ such that 
	$[ e_{\alpha}, e_{-\alpha} ]=h_{\alpha}$. The elements 
\begin{align}\label{chevalleybasis}
	\{e_{\alpha}, \alpha \in \Phi;\; h_{\delta},\delta \in \Delta\}
\end{align}
form a basis for $\mathcal{L}$, called a \textbf{Chevalley basis}.
	
The map $\mathrm{ad}e_{\alpha}$ is a nilpotent linear map on 
$\mathcal{L}$. For $t \in \mathbb{C}$, the map  $\mathrm{ad}(te_{\alpha})=
t(\mathrm{ad}e_{\alpha})$ is also nilpotent.
Thus $\mathrm{exp}(t(\mathrm{ad}e_{\alpha}))$ is an automorphism of 
$\mathcal{L}$. 
We denote by $\mathcal{L}_{\mathbb{Z}}$ the subset of $\mathcal{L}$ of 
all $\mathbb{Z}$-linear combinations of the Chevalley basis elements of $\mathcal{L}$.
Thus $\mathcal{L}_{\mathbb{Z}}$ is a Lie algebra over $\mathbb{Z}$. We define $\mathcal{L}_{k}:=\mathcal{L}_{\mathbb{Z}}\otimes_{\mathbb{Z}} k$.
Then $\mathcal{L}_{k}$ is a Lie algebra 
over $k$ with respect to the natural Lie multiplication.

So we are in a position to define the 
Chevalley groups of adjoint type or the elementary Chevalley groups over any field $k$.
The \textit{adjoint Chevalley group} of type $\Phi$ over the field $k$, denoted by 
$G(\Phi, k)$ or simply $G$, is defined to be the subgroup of automorphism group of the Lie 
algebra $\mathcal{L}_k$ generated by $x_\alpha(t):=\mathrm{exp}(t(\mathrm{ad}e_{\alpha}))$ for all 
$\alpha \in \Phi, t \in k$. In fact, the group $G$ over $k$ 
is determined up to isomorphism by the simple Lie algebra 
$\mathcal{L}$ over $\mathbb{C}$ and the field $k$. Following the notations as in \cite[Theorem 12.1.1]{ca} (assume $\Phi\neq A_1$), let $\widetilde{G}$ be the abstract group generated by $\widetilde{x}_{\alpha}(t)$ ($\alpha\in \Phi, t\in k$) satisfying the following relations:
\begin{enumerate}[leftmargin=*]
\item $\widetilde{x}_{\alpha}(t)\widetilde{x}_{\alpha}(s)=\widetilde{x}_{\alpha}(t+s)$\; ($t,s\in k$),
\item $[\widetilde{x}_{\beta}(s),\widetilde{x}_{\alpha}(t)]
=\displaystyle\prod_{\substack{i,j>0\\i\alpha+j\beta\in \Phi}}\widetilde{x}_{i\alpha+j\beta}(C_{ij,\alpha\beta}(-t)^is^j)$\; ($\alpha,\beta\in \Phi$ and $t,s\in k$),
\item $\widetilde{h}_{\alpha}(t)\widetilde{h}_{\alpha}(s)=\widetilde{h}_{\alpha}(ts)$ \;($t,s\in k^{\times}$),
\end{enumerate}
where $\widetilde{h}_{\alpha}(t)=\widetilde{n}_{\alpha}(t)\widetilde{n}_{\alpha}(-1)$, 
 $\widetilde{n}_{\alpha}(t)=\widetilde{x}_{\alpha}(t)\widetilde{x}_{-\alpha}(-t^{-1})\widetilde{x}_{\alpha}(t)$ and $C_{ij,\alpha\beta}$ are certain integers. Then
 $\widetilde{G}/Z(\widetilde{G})\cong G$ (the adjoint Chevalley group). 
The group $\widetilde{G}$ is called  \textit{universal Chevalley group} of type $\Phi$ over $k$. For detail constructions of $\widetilde{G}$ see \cite[Chapter 3]{St2}.
If $Z$ is any subgroup of $Z(\widetilde{G})$, then the factor group $\widetilde{G}/Z$ is called \textit{Chevalley group}. For $\alpha\in \Phi, t\in k^{\times}$ set
\begin{align}
n_{\alpha}(t):&=x_{\alpha}(t)x_{-\alpha}(-t^{-1})x_{\alpha}(t) \\
h_{\alpha}(t):&=n_{\alpha}(t)n_{\alpha}(-1)\label{torus}.
\end{align}
\begin{notation}\label{uvhn}
	Let $G$ be an adjoint Chevalley group. For each $\alpha\in \Phi$, let $w_{\alpha}$ be the reflection in the hyperplane orthogonal to $\alpha$. The following notations shall be used throughout this paper.
\begin{align*}
U&=\langle x_{\alpha}(t)\mid \alpha \in \Phi^{+}, t\in k\rangle, \\
V&=\langle x_{\alpha}(t)\mid \alpha \in \Phi^{-}, t\in k\rangle, \\
H&=\langle h_{\alpha}(t)\mid \alpha \in \Delta, t\in k^{\times}\rangle, \\
N&=\langle H, n_{\alpha}(1)\mid \alpha \in \Phi\rangle,\\
W&=\langle w_{\alpha} \mid \alpha\in \Delta \rangle. 
\end{align*}
\end{notation}
Observe that $n_{\alpha}(1)Hn_{\alpha}(1)^{-1}=H$, i.e., $H$ is a normal subgroup of $N$ and $N/H\cong W$ given by $n_{\alpha}(1)H\mapsto w_\alpha$. 
\begin{example}\label{example1}
Let $\mathcal{L}_k=\mathfrak{sl}_2(k)=\left\{\begin{pmatrix}
a&b\\c&d\end{pmatrix}\in \mathrm{M}_2(k)\mid a+d=0\right\}$ be the simple Lie algebra of type $A_1$. A Chevalley basis for $\mathfrak{sl}_2(k)$ is $$\left\{e_\alpha=\begin{pmatrix}0&1\\0&0\end{pmatrix}, e_{-\alpha}=\begin{pmatrix}0&0\\1&0\end{pmatrix}; h_\alpha=\begin{pmatrix}1&0\\0&-1\end{pmatrix}\right\}.$$
\end{example}
Now $\widetilde{x}_{\alpha}(t):=\exp(te_{\alpha})=\begin{pmatrix}1&t\\0&1\end{pmatrix}$ and $\widetilde{x}_{-\alpha}(t):=\exp(te_{-\alpha})=\begin{pmatrix}1&0\\t&1\end{pmatrix}$. Then
\[
\widetilde{n}_{\alpha}(t):=\widetilde{x}_{\alpha}(t)\widetilde{x}_{-\alpha}(-t^{-1})\widetilde{x}_{\alpha}(t)=\begin{pmatrix}0&t\\-t^{-1}&0\end{pmatrix},\;
\widetilde{h}_{\alpha}(t):=\widetilde{n}_{\alpha}(t)\widetilde{n}_{\alpha}(-1)=\begin{pmatrix}t&0\\0&t^{-1}\end{pmatrix}.\]
Note that $\SL_2(k)=\langle\widetilde{x}_{\alpha}(t),\widetilde{x}_{-\alpha}(t)\mid t\in k\rangle$, and the adjoint Chevalley group 
$G(A_1, k)=\langle x_{\alpha}(t), x_{-\alpha}(t)\mid t\in k\rangle$.  Thus there is a surjective homomorphism $\eta: \SL_2(k)\rightarrow G(A_1,k)$ given by $\eta(\widetilde{x}_{\alpha}(t))=x_{\alpha}(t),\; \eta(\widetilde{x}_{-\alpha}(t))=x_{-\alpha}(t),\; \eta(\widetilde{n}_{\alpha}(t))=n_{\alpha}(t)$ and $\eta(\widetilde{h}_{\alpha}(t))=h_{\alpha}(t)$ with $\ker(\eta)=\{\pm I_2\}$. Hence $G(A_1,k)\cong \mathrm{PSL}_2(k)$. 
\begin{definition}\label{hhat}
	Let $R:=\mathbb{Z}\Phi$ be the root lattice and $\chi:R \rightarrow k^{\times}$ be a $k$-character (i.e., a group homomorphism from the additive group $R$ to the multiplicative group $k^{\times}$). Then $\chi$ gives rise to an automorphism, denoted by $h(\chi)$, of the Lie algebra $\mathcal{L}_k$ given by $h(\chi)(e_\alpha)=\chi(\alpha)e_{\alpha} \;(\alpha\in \Phi)$ and $h(\chi)(h_{\delta})=h_{\delta} \;(\delta\in \Delta)$. Here $e_\alpha, h_{\delta}$ as in equation \eqref{chevalleybasis}. Define  $\widehat{H}:=\{h(\chi)\mid \chi:\mathbb{Z}\Phi \rightarrow k^{\times}\}$.
\end{definition}
For $\alpha\in \Delta, t\in k^{\times}$ the element $h_{\alpha}(t)$ corresponds to the character $\chi:=\chi_{\alpha, t}$ given by $\chi_{\alpha, t}(\beta)=t^{\frac{2(\alpha, \beta)}{(\alpha, \alpha)}}$, where $\frac{2(\alpha, \beta)}{(\alpha, \alpha)}\in \mathbb{Z}$ are the Cartan integers (for all $\beta\in \Phi$). Therefore $H\leq \widehat{H}$.
In other words, $H$ can be viewed as a subgroup of $G$ consisting of all automorphisms $h(\chi)$ of $\mathcal{L}_k$ for which a $k$-character $\chi$ of $\mathbb Z\Phi$ can be extended to a $k$-character of $\mathbb{Z}\langle q_1, q_2, \ldots, q_l\rangle $ (\textit{weight lattice}), where $q_1, \ldots, q_l$ are fundamental weights. Note that $\widehat{H}\leq N_{\Aut(\mathcal{L}_k)}(G)$. (For details, see \cite[Chapter 7]{ca}).
\begin{example}
For $\SL_2(k)$, we have $H=\{\diag(t,t^{-1})\mid t\in k^{\times}\}\cong k^{\times}$. In this special case, let $\Delta=\{\alpha\}$ and $\Phi=\{\alpha, -\alpha\}$ (as rank is $1$). Therefore the root lattice is  $R=\mathbb{Z}\langle\alpha\rangle\cong \mathbb{Z}$ and the weight lattice is (say) $P=\mathbb{Z}\langle\alpha/2\rangle\cong \mathbb{Z}$. 
Now let $\chi$ be a $k$-character of $R$ given by $\chi(\alpha)=t$ ($t\in k^{\times}$). Then $\chi(\alpha/2)^2=\chi(\alpha/2+\alpha/2)=\chi(\alpha)=t$. Hence, if $k=\Q$ (resp. $k=k_0(T), T$ transcendental over $k_0\subset k$) then $H\neq\widehat{H}$ since not all $\Q$-character $\chi$ (resp. $k_0(T)$-character) can be extended to its weight lattice $P$ as $\Q$-character (resp. $k_0(T)$-character). But if $k=\C$ then   $H=\widehat{H}$.
\end{example}
\subsection{Twisted Chevalley groups}\label{tchev}
Every semisimple linear group of classical type can be thought of as a Chevalley group or as  a twisted Chevalley group. In general, the classical simple groups are the special linear, symplectic, special orthogonal and special unitary groups corresponding to forms whose Witt index is sufficiently large. The twisted groups were discovered independently by Steinberg 
(cf. \cite{steinberg59}) and Tits (cf. \cite{tits}). For our exposition, we will follow Steinberg's approach. 
 The twisted Chevalley groups can be obtained as certain subgroups of the Chevalley groups $G=G(\Phi, k)\leq\Aut(\mathcal{L}_k)$, where $\mathcal{L}$ is a Lie algebra over $\mathbb{C}$ and $\mathcal{L}_k$ is the corresponding Lie algebra over $k$.
We refer the reader to \cite{ca} and \cite{St2} for details.

\begin{definition}[Graph automorphism]
A symmetry of the Dynkin diagram induces this type of automorphisms. A symmetry of the Dynkin diagram of $\mathcal{L}$ is a permutation $\rho$ of the nodes of the diagram such that the number of edges joining nodes $i, j$ is the same as the number of edges joining nodes $\rho(i), \rho(j)$ for all $i\neq j$, i.e., if $n_{ij}$ is equal to the number of edges joining the nodes corresponding to $\alpha_i, \alpha_j \in \Delta$ (simple roots), then $n_{ij}=n_{\rho(i)\rho(j)}$. Any graph automorphism will be denoted by $\overline{\rho}$. 
\end{definition}
\begin{definition}[Field automorphism]
Let $f$ be an automorphism of the field $k$, then the map 
\[\widetilde{f}:G\rightarrow G \text{ defined by }\widetilde{f}(x_{\alpha}(t))=x_{\alpha}(f(t)),\;  \alpha \in \Phi, t\in k\] can be extended to an automorphism of $G$. The automorphisms obtained in this way are called field automorphisms. We shall abuse the notation slightly and denote the field automorphism of $G$ by $f$ itself.
\end{definition}
Since we are considering only  $\mathrm{char}\;k=0$, then non-trivial graph automorphisms exist only for the root systems of types 
$A_l \, (l\geq 2)$, $D_l\, (l\geq 4), E_6$ and $D_4$, and the order of the graph automorphisms are $2, 2, 2$ and $3$, respectively.	
Let $\rho$ be a non-trivial symmetry of the Dynkin diagram of $\mathcal{L}$ then the order of $\rho$ is $2$ or $3$. For this article, by a Chevalley group of type $\Phi$ we always mean one of the following types: $A_l \, (l\geq 2)$, $D_l\, (l\geq 4)$ and $E_6$.

Let $G$ be the Chevalley group of type $\Phi$.   
Then there is a graph automorphism $\overline{\rho}$ of $G$ such that 
\[\overline{\rho}(x_{\alpha}(t))=x_{\rho(\alpha)}(t),\] for all $\alpha \in \Delta$ and $t\in k$. Let $f: G\rightarrow G$ be a field  automorphism, i.e., $f(x_{\alpha}(t))=x_{\alpha}(f(t))$. Then 
\[(\overline{\rho} f)(x_{\alpha}(t))=\overline{\rho}(x_{\alpha}(f(t)))=x_{\rho(\alpha)}(f(t))=f(x_{\rho(\alpha)}(t))=f(\overline{\rho}(x_{\alpha}(t)))=(f \overline{\rho})(x_{\alpha}(t))\]
for all $\alpha \in \Delta$ and $t\in k$. Therefore $f\circ \overline{\rho}=\overline{\rho}\circ f$, as $G$ is generated by $x_{\alpha}(t)$. 
Let $n$ be the order of $\rho$, then 
\[\overline{\rho}^n\cdot x_{\alpha}(t)=
x_{\rho^n(\alpha)}(t)=x_{\alpha}(t)\] 
for all $\alpha \in \Delta$. Therefore $\overline{\rho}^n=\mathrm{Id}$.
Hence for the automorphism $\sigma :=\overline{\rho}\circ f: G\rightarrow G$, we have  $\sigma^n=\overline{\rho}^nf^n=f^n$. If $f$ is any non-trivial field automorphism such that $f^n=\mathrm{Id}$, then $\sigma^n=\mathrm{Id}$. With the terminologies as in Notation \ref{uvhn}, we have the following important result which will be used to describe the twisted Chevalley group. 
\begin{lemma}\cite[Proposition 13.4.1]{ca}
	Let $G$ be a Chevalley group of type $\Phi$ over a field $k$ of $\mathrm{char}\;k=0$, whose Dynkin diagram has a non-trivial symmetry $\rho$. Let $\overline{\rho}$ be the graph automorphism corresponding to $\rho$ and $f$ be a non-trivial field automorphism chosen such that $\sigma=\overline{\rho} f$ satisfies $\sigma^n=\mathrm{Id}$, i.e., $f$ is chosen so that $f^n=\mathrm{Id}$, where $n$ is the order of $\rho$ which is either $2$ or $3$. Then we have $\sigma(U)=U, \sigma(V)=V, \sigma(H)=H, \sigma(N)=N$
and $\sigma: N/H\cong W \rightarrow W$ is given by $\sigma(w_{\alpha})=w_{\rho({\alpha})}$ for all $\alpha \in \Delta$. Here $w_{\alpha}$ denotes the reflection in the hyperplane orthogonal to $\alpha$.
\end{lemma}
Now we are in a position to define the twisted Chevalley groups as a certain subgroups of the Chevalley groups which are fixed elementwise by the automorphisms $\sigma=\overline{\rho}f$, where $\overline{\rho}$ is a graph automorphism and $f$ is a field automorphism of $G$.
\begin{definition} The following notations will be used henceforth. 
\begin{enumerate}
		\item $U':=\{u\in U\mid \sigma(u)=u\}$.
		\item $V':=\{v\in V\mid \sigma(v)=v\}$.
		\item $G'_{\sigma}:=\langle U', V'\rangle\leq G$.
		\item $H':=H\cap G'_{\sigma}$.
		\item $N':=N\cap G'_{\sigma}$.
	\end{enumerate}
\end{definition}
The group $G'_{\sigma}$ is called the \textbf{twisted Chevalley group of adjoint type} with respect to the automorphism $\sigma$. Similarly, we can define $\widehat{H'}:=\widehat{H}\cap N_{\Aut(\mathcal{L})}(G'_{\sigma})$, which contains $H'$ and normalizes $G'_{\sigma}$ (cf. \defref{hhat}). This will be useful to describe the diagonal automorphisms of the twisted Chevalley group in the next section. In this context, we recall an important result by Steinberg, which says that all the twisted Chevalley groups of adjoint type are simple (with few exceptions), 
see \cite[p. 884, Theorem 8.1]{steinberg59}. 
Also, we can define the universal twisted Chevalley group $\widetilde{G}'_{\sigma}$. Generally twisted Chevalley groups are of the form $\widetilde{G}'_{\sigma}/Z$ (also denoted by $G'_{\sigma}$), where $Z\leq Z(\widetilde{G}'_{\sigma})$ is a central subgroup (cf. \secref{section21}). Then we have the following short exact sequence of groups
\begin{align}\label{qtog}
\xymatrix{1\ar[r]&Z(G'_{\sigma})\ar[r]&G'_{\sigma}\ar[r]&G'_{\sigma}/Z(G'_{\sigma})\ar[r]&1}
\end{align}
 where $G'_{\sigma}/Z(G'_{\sigma})$ is the twisted Chevalley group of adjoint type. 
\begin{remark}\label{center}
It follows from \cite[p. 29, Lemma 28 (d)]{St2} that the center $Z(\widetilde{G})$ of the universal Chevalley group $\widetilde{G}$ is finite. Also, observe that the center of the twisted Chevalley group of universal type is  $Z(\widetilde{G}'_{\sigma})=Z(\widetilde{G})^{\sigma}=\{g\in Z(\widetilde{G})\mid \sigma(g)=g\}$ (see \cite[p. 108, Exercise]{St2}). Thus the center $Z(G'_{\sigma})$ of the twisted Chevalley group is finite. 
\end{remark}
\begin{example}
\begin{enumerate}\addtolength{\itemindent}{-6mm}
\item\cite[p. 882, Section 6]{steinberg59} Let $\mathcal{L}=A_l\; (l\geq 2)$. Then the adjoint Chevalley group is $G\cong \mathrm{PSL}_{l+1}(k)$ and the universal Chevalley group is $\widetilde{G}\cong \mathrm{SL}_{l+1}(k)$. The twisted Chevalley group of adjoint type is 
 $G'_{\sigma}\cong \mathrm{PSU}_{l+1}(k,J)$ and the universal twisted Chevalley group is $\widetilde{G}'_{\sigma}\cong \mathrm{SU}_{l+1}(k,J)$, where 
\[J=\epsilon\begin{pmatrix} &&&&1\\ &&&-1&\\&&1&&\\&-1&&&\\\reflectbox{$\ddots$}\end{pmatrix}.\]
Here $\epsilon \in k$ such that $\epsilon+\bar{\epsilon}=0$ if $l$ is odd and $\epsilon=1$ if $l$ is even. Here $\bar\;:k\rightarrow k$ is an involutory automorphism of $k$ with fixed field $k_0$, i.e., $k$ is a degree two Galois extension of $k_0$. (A prototypical example is the Galois extension $\mathbb{C}$ over $\R$ with `bar' being the complex conjugate). 
\item\cite[p. 886, Section 9]{steinberg59} Let $\mathcal{L}=D_l\; (l\geq 5)$, then the adjoint Chevalley group is
 $G\cong \mathrm{P\Omega}_{2l}(k, B_D)$ and the universal Chevalley group is $\widetilde{G}\cong \mathrm{\Omega}_{2l}(k, B_D)$, where $B_D=x_1x_{-1}+\cdots+x_lx_{-l}$ (quadratic form over $k$). The adjoint twisted Chevalley group is 
$G'_{\sigma}\cong \mathrm{P\Omega}_{2l}(k_0, B)$ and the universal twisted Chevalley group is $\widetilde{G}'_{\sigma}\cong \mathrm{\Omega}_{2l}(k_0, B)$, where 
$B=x_1x_{-1}+\cdots+x_{l-1}x_{-(l-1)}+(x_l-dx_{-l})(x_l-\bar{d}x_{-l})$ (quadratic form over $k_0$), and  $k=k_0(d)$. Here $\bar\; : k\rightarrow k$ is an order two automorphism of $k$ with $k_0$ being its fixed field. 
\end{enumerate}
\end{example}
\subsection{Automorphisms of twisted Chevalley groups}
To study the $R_{\infty}$-property of a group, it is important to understand the automorphisms of the given group. In this section, we recall three fundamental types of automorphism of the twisted Chevalley group. We refer the reader to \cite[Section 12.2]{ca} for details.
\begin{definition}[Inner automorphism]
Given any $x\in G'_{\sigma}$, we define 
\[i_x:G'_{\sigma}\rightarrow G'_{\sigma} \text{ given by }i_x(g)=xgx^{-1}\] is an automorphism of $G'_{\sigma}$. The automorphism $i_x$ is called the inner automorphism of $G'_{\sigma}$ induced by $x$. 
\end{definition}
\begin{definition}[Diagonal automorphisms]
Let $h\in \widehat{H'}\setminus H'$. Then the map 
\[d_h:G'_{\sigma}\rightarrow G'_{\sigma} \text{ given by }d_h(g)=hgh^{-1}\] is an automorphism of $G'_{\sigma}$, since $G'_{\sigma}$ is normalized by $\widehat{H'}$. The automorphism $d_h$ is called a diagonal automorphism. The diagonal automorphisms of $G'_{\sigma}$ are obtained by conjugating with suitable diagonal matrices.
\end{definition}
\begin{definition}[Field automorphism]
Let $f$ be an automorphism of the field $k$, then the map 
\[\overline{f}:G'_{\sigma}\rightarrow G'_{\sigma} \text{ given by }\overline{f}(x_{\alpha}(t))=x_{\alpha}(f(t)),\;  \alpha \in \Phi, t\in k\] can be extended to an automorphism of $G'_{\sigma}$. The automorphisms obtained in this way are called field automorphisms. In terms of matrices this amount to replacing each term of the matrix by its image under $f$. \par
Now we are ready to state the following theorem due to Steinberg that is the main tool to prove our results. 
\begin{theorem}\cite[p. 111, Theorem 36]{St2}\label{staut}
With the preceding notations, let $\sigma=\overline{\rho}f(\neq \mathrm{Id})$ be the automorphism of the Chevalley group $G$. Then every automorphism $\varphi$ of the twisted Chevalley group $G_{\sigma}'$ is a product of an inner automorphism $i_g$, a diagonal automorphism $d_h$ and a field automorphism $\overline{f}$, i.e., $\varphi=\overline{f}\circ d_h \circ i_g$ for some $g\in G_{\sigma}'$ and $h\in \widehat{H'}$.
\end{theorem}
\end{definition}
The following lemma will be used in the proof of our main theorems. 
\begin{lemma}\label{normality}
	Let $D$ be the group of all diagonal automorphisms of $G'_{\sigma}$ and $\Gamma$ be the group generated by all field and diagonal automorphisms of $G'_{\sigma}$, then $D$ is a normal subgroup of $\Gamma$. 
\end{lemma}
\begin{proof}
Clearly, $D$ is a subgroup of $\Gamma$. Let $\overline{f}$ be a field automorphism of $G'_{\sigma}$ and $d_h\in D$, where $h=\diag(h_1,  \ldots, h_{|\Delta|+|\Phi|})\in \widehat{H'}\setminus H', h_i\in k^{\times}$. Suppose $g=(g_{ij})\in G'_{\sigma}$, then we have 
 \begin{align*}
\overline{f}d_h\overline{f}^{-1}(g_{ij})&=\overline{f}d_h(f^{-1}(g_{ij}))=\overline{f}h(f^{-1}(g_{ij}))h^{-1}
=\overline{f}(h_if^{-1}(g_{ij})h_j^{-1})\\
&=(f(h_if^{-1}(g_{ij})h_j^{-1}))=(f(h_i)g_{ij}f(h_j^{-1}))=d_{\widetilde{h}}(g_{ij})
\end{align*}
where $\widetilde{h}=\diag(f(h_1), \ldots, f(h_{|\Delta|+|\Phi|}))$. Hence $D$ is a normal subgroup of $\Gamma$.
\end{proof}

\subsection{Some useful results}
The following two lemmas hold for arbitrary groups. 
Suppose $\varphi$ is an automorphism of a group $G$. Let $i_g$ be the inner automorphism of $G$ for some $g$ in $G$, i.e., $i_g(x)=gxg^{-1}$ for all $x\in G$. Let $\mathcal{R}(\varphi):=\{[g]_{\varphi}\mid g\in G \}$. Thus the Reidemeister number $R(\varphi)$ is the cardinality of $\mathcal{R}(\varphi)$. 	Now let $x,y\in G$ such that $[x]_{\varphi\circ i_g}=[y]_{\varphi\circ i_g}$. Then there exists a $z\in G$ such that
$$y=zx(\varphi\circ i_g)(z^{-1})=zx\varphi(gz^{-1}g^{-1})=zx\varphi(g)\varphi(z^{-1})\varphi(g^{-1}).$$ This implies that 
$y\varphi(g)=zx\varphi(g)\varphi(z^{-1})$, i.e, $[x\varphi(g)]_{\varphi}=[y\varphi(g)]_{\varphi}$. 
Thus we get a well-defined map \[\widehat{\varphi}: \mathcal{R}(\varphi\circ i_g)\longrightarrow \mathcal{R}(\varphi)\] given by  $\widehat{\varphi}([x]_{\varphi\circ i_g})=[x\varphi(g)]_{\varphi}$. The map $\widehat{\varphi}$ is bijective as well. We summarise this as follows:
\begin{lemma}\cite[Corollary 3.2]{FLT}\label{inner}
Suppose $\varphi \in \Aut(G)$ and $i_g \in \mathrm{Inn}(G)$ then $R(\varphi i_g)=R(\varphi)$. In particular, $R(i_g)=R(\mathrm{Id})$, i.e., the number of inner twisted conjugacy classes in $G$ is equal to the number of conjugacy classes in $G$.
\end{lemma}
The following result appears in \cite[Lemma 2.2]{MS}. We include a proof for the sake of completeness. 
\begin{lemma}\label{msqtog}
	Let $1\rightarrow N\overset{i}{\rightarrow} G\overset{\pi}{\rightarrow} Q\rightarrow 1$ be an exact sequence of groups. Suppose that $N$ is a characteristic subgroup of $G$, i.e., $\varphi(N)=N$ for all $\varphi\in \Aut(G)$. If $Q$ has the $R_{\infty}$-property, then $G$ also has the $R_{\infty}$-property.
\end{lemma}
\begin{proof}
Suppose that $\varphi$ is any automorphism of $G$. Since $\varphi(N)=N$ then  $\varphi$ induces an automorphism $\overline{\varphi}$ of $Q\cong G/N$ 
such that the  following diagram commutes:
\[\xymatrix{
	1\ar[r] & N\ar[r]^{i}\ar[d]_{\varphi|_{N}}& G\ar[r]^{\pi} \ar[d]^{\varphi}& Q\ar[r]\ar[d]^{\overline{\varphi}}&1 \\
	1\ar[r] & N\ar[r]^{i}& G\ar[r]^{\pi}& Q\ar[r]&1 
}\] 
where $\varphi|_{N}$ is the automorphism of $N$. In particular, we have $\overline{\varphi}\circ \pi=\pi\circ \varphi$.
Now, observe that $\pi$ induces a surjective map $\widehat{\pi}:\mathcal{R}(\varphi)\longrightarrow\mathcal{R}(\overline{\varphi})$ given by $\widehat{\pi}([x]_{\varphi})=[\pi(x)]_{\overline{\varphi}}$ for all $x\in G$. Hence $R(\varphi)\geq R(\overline{\varphi})$. Thus $G$ has the $R_{\infty}$-property since $Q$ has so.
\end{proof}
Let $\Q$ be the field of rational numbers, $S$ be the set of prime numbers and $2^{S}$ be the set of all subsets of $S$. Then define the following map 
\[\nu:\Q \rightarrow 2^{S}\] by $\nu(a/b)=\{\text{all prime divisors of }a\}\cup \{\text{all prime divisors of b}\}$, where $a$ and $b$ are mutually prime integers. Using the map $\nu$, Nasybullov \cite[Lemma 2.5]{FN} proved the following result.
\begin{lemma}\label{lemma5}
Let $k$ be a field of characteristic zero such that the transcendence degree of $k$ over $\Q$ is finite. If the automorphism $f$ of the field $k$ acts on the elements $z_1, z_2, \ldots$ of the field $k$ by the rule
\[f:z_i\mapsto \alpha a_iz_i, \]
where $\alpha\in k$, $1\neq a_i\in \Q\subset k$ and $\nu(a_i)\cap \nu(a_j)=\emptyset$ for $i\neq j$, then there are only a finite number of non-zero elements among $z_1, z_2, \ldots$. 
\end{lemma}
 We include a proof of the following lemma which can be found in \cite[Lemma 6]{NaT}.
\begin{lemma}\label{lemma6}
	Let $R$ be an integral domain and $M$ be an infinite subset of $R$. Let $f(T)$ be a non-constant rational function with coefficients from the ring $R$. Then the set $P=\{f(a)\mid a\in M\}$ is infinite.
\end{lemma}
\begin{proof}
Let $f(T)=\frac{g(T)}{h(T)}$, where $g(T),h(T)\in R[T]$ with $h(T)\neq 0$. If possible, suppose that the set $P$ is finite. That is for an infinite subset $\{a_i\mid i=1,2,\dots\}$ of $M$ we have $f(a_i)=c\in R$ (say). In that case, the polynomial $\alpha(T)=g(T)-ch(T)\in K[T]$ has infinitely many roots in $K$, where $K$ is the field of fraction of $R$. Thus $\alpha(T)=0$ and hence, $f(T)=c$ is a constant function, a contradiction.
\end{proof}
 The proof of the following result is similar to the Chevalley group case as in \cite[Lemma 7]{NaT} and hence we skip the details.
\begin{lemma}\label{lemma7}
Let $g(T)=h_{\alpha_{1}}(T)h_{\alpha_{2}}(T)\cdots h_{\alpha_{l}}(T)$ (here $h_{\alpha_i}(T)$ as in equation \eqref{torus}) be an element of the twisted Chevalley group over $k(T)$ and $\chi: \mathbb{Z}\Phi \rightarrow k^{\times}$ be a homomorphism. Then for any $m$, the  element $g(T)^mh(\chi)$ with respect to the Chevalley basis has a diagonal form such that its trace belongs to $k(T)\setminus k$. Here $k(T)$ denotes the field of rational functions with one variable $T$ over the field $k$. 
\end{lemma}

\section{Proofs of the main results}\label{mainsection}
We are now ready to establish our two main theorems.
\subsection{Proof of \thmref{mainthm1}}\label{proof1}
The argument in the case of Chevalley groups may be followed here verbatim. In view of equation \eqref{qtog} and \lemref{msqtog}, it is enough to prove this theorem for twisted Chevalley group of adjoint type $G'_{\sigma}$. 
Let $\varphi \in \Aut(G'_{\sigma})$, then by  \thmref{staut} we have $\varphi=\overline{f}\circ d_h\circ i_g$, where $i_g$ is an inner automorphism for some $g\in G'_{\sigma}$, $d_h$ is a diagonal automorphism for some $h\in \widehat{H'}$ and $\overline{f}$ is a field automorphism. By Lemma \ref{inner}, we may assume that $\varphi=\overline{f}\circ d_h$. In view of Lemma \ref{normality}, we have $\varphi=d_{\widetilde{h}}\circ \overline{f}$ for some $\widetilde{h}\in \widehat{H'}$. 
	
\noindent
\textbf{Claim:} $R(\varphi)=\infty$. 
	
Suppose if possible $R(\varphi)<\infty$. Let
\begin{align*}
g_i=h_{\alpha_1}(p_{i1})h_{\alpha_2}(p_{i2})\cdots h_{\alpha_l}(p_{il}),
\end{align*}
where $p_{11}<p_{12}<\cdots <p_{1l}<p_{21}<p_{22}<\cdots$ are primes and $\{\alpha_{1}, \alpha_{2}, \ldots, \alpha_l\}=\Delta$ is the set of all simple roots. Now writing $g_i$ with respect to the aforementioned Chevalley basis \eqref{chevalleybasis}, we get 
	\[g_i=\diag(a_{i1}, a_{i2}, \ldots, a_{i{|\Phi|}};\underbrace{1, \ldots, 1}_l),\] 
	where $a_{ij}\in \Q$ such that $\nu(a_{ij})\neq \emptyset$ and $\nu(a_{ij})\cap \nu(a_{rs})=\emptyset$ for $i\neq r$, since $\nu(a_{ij})\subset \{p_{i1}, \ldots, p_{il}\}$ for all $i, j$. Since $a_{ij}\in \Q$ and field automorphism acts identically on the prime subfield $\Q$ of $k$, then $\overline{f}(g_i)=g_i$. Also, the diagonal automorphism $d_h$ acts by conjugation. Here $h$ is an $(|\Phi|+|\Delta|)\times (|\Phi|+|\Delta|)$-diagonal matrix and $g_i$ is diagonal as well, so $d_h(g_i)=g_i$ for all $i$.
Therefore $\varphi(g_i)=(\overline{f}\circ d_h)(g_i)=g_i$ for all $i$. Since we are assuming that the number of $\varphi$-twisted conjugacy classes is finite, without loss of generality, we may also assume that $g_i\sim_{\varphi} g_1$ for all $i=2, 3, \ldots$. Therefore 
	\begin{align*}
	g_1&=z_ig_i\varphi(z_i^{-1})=z_ig_i(\overline{f}d_h)(z_i^{-1})=z_ig_i(d_{\widetilde{h}}\overline{f})(z_i^{-1})=z_ig_i\widetilde{h}\overline{f}(z_i^{-1})
	\widetilde{h}^{-1}
	\end{align*}
	for some $z_i\in G'_{\sigma}$ for all $i=2,3, \ldots$. Hence $g_1\widetilde{h}=z_i(g_i\widetilde{h})
	\overline{f}(z_i^{-1})$.
	This implies that
	\begin{align}\label{field}
	\overline{f}(z_i)&=(g_1\widetilde{h})^{-1}z_i(g_i\widetilde{h}).
	\end{align}
	Let \[\widetilde{h}=\diag(b_1, b_2, \ldots, b_{|\Phi|}; \underbrace{1, \ldots, 1}_l)\in \widehat{H'},\] 
	then we have 
\begin{align*}
	g_i\widetilde{h}=\diag(a_{i1}b_1, a_{i2}b_2, \ldots, a_{i|\Phi|}b_{|\Phi|}; \underbrace{1, \ldots, 1}_l)
\end{align*}
for all $i=1,2,\ldots $.
Let $z_i=\begin{pmatrix}Q_i&R_i\\S_i&T_i\end{pmatrix}$ be a block matrix, where 
$Q_i=(q_{i, mn})_{|\Phi|\times |\Phi|}$, $R_i=(r_{i, mn})_{|\Phi|\times |\Delta|}, S_i=(s_{i, mn})_{|\Delta|\times |\Phi|}$ and $T_i=(t_{i, mn})_{|\Delta|\times |\Delta|}$. 
	Then by equation \eqref{field}, for all $m, n=1, 2, \ldots, |\Phi|$, we have 
	\[f(q_{i, mn})=a_{1m}^{-1}b_m^{-1}a_{in}b_nq_{i, mn}=c_{mn}a_{in}q_{i, mn},\] where $c_{mn}=(a_{1m}b_mb_n^{-1})^{-1}$. Since $\nu(a_{in})\neq \emptyset$ and $\nu(a_{in})\cap \nu(a_{jn})=\emptyset$ for $i\neq j$, then we can apply \lemref{lemma5} to the elements $q_{2, mn}, q_{3, mn}, \ldots$ and we get $q_{j, mn}=0$ for all $j>N_{mn}$ for some integer $N_{mn}$. Choose $N=\mathrm{max}_{m, n=1, 2, \ldots,  |\Phi|}N_{mn}$,  then we have $Q_j=(q_{j, mn})_{|\Phi|\times |\Phi|}=O_{|\Phi|\times |\Phi|}$ for all $j>N$ (by $O_{m\times n}$, we mean the $m\times n$ matrix all of whose entries are zero). Using the similar arguments to the matrices $\{S_i\}_{i}$, we conclude that all the matrices $S_i$ reduces to the zero matrix for sufficiently large indices $i$. Hence the matrix $z_i$ has the following form for sufficiently large indices $i$:
\[z_i=\begin{pmatrix}O_{|\Phi|\times |\Phi|}&R_i\\O_{|\Delta|\times |\Phi|}&T_i\end{pmatrix}.\]
Therefore, for $i$ sufficiently large, $\mathrm{det}(z_i)=0$, which is a contradiction as $z_i\in G'_{\sigma}\leq \Aut(\mathcal{L}_k)$. 
Hence $R(\varphi)=\infty$. 

This completes the proof. 
\subsection{Proof of \thmref{mainthm2}}
In view of equation \eqref{qtog} and \lemref{msqtog}, it is enough to show this theorem for twisted Chevalley group of adjoint type $G'_{\sigma}$. 
Let $\varphi \in \Aut(G'_{\sigma})$, then as in the proof of Theorem 1.2 in \secref{proof1}, we may assume that $\varphi=\overline{f}\circ d_h$. Since $\Aut(k)$, the automorphism group of the field $k$, is periodic, the field automorphism $\overline{f}$ is of finite order, say $n$. 
	
	\noindent
	\textbf{Claim:} $R(\varphi)=\infty$. 
	
	Suppose if possible $R(\varphi)<\infty$. Let $g(T)=h_{\alpha_{1}}(T)h_{\alpha_{2}}(T)\cdots h_{\alpha_{l}}(T)$ be an element of the twisted Chevalley group over $k(T)$, where $\alpha_1, \ldots, \alpha_l$ are the simple roots. Consider the elements $g_i:=g(x_i)$, where $\{x_i\}$ is an infinite set of non-zero rational numbers. 
	Since we are assuming that the number of $\varphi$-twisted conjugacy classes is finite, without loss of generality, we may also assume that $g_i\sim_{\varphi} g_1$ for all $i=2, 3, \ldots$.
Therefore 
	\[g_i=z_ig_1\varphi({z_i^{-1}})\] for some $z_i\in G'_{\sigma}$ for all $i=2,3, \ldots $. Then we have 
	\begin{equation}\label{keyeq}
	\begin{split}
	g_i&=z_ig_1\varphi(z_i^{-1})\\
	\varphi(g_i)&=\varphi(z_i)\varphi(g_1)\varphi^2(z_i^{-1})\\
	\varphi^2(g_i)&=\varphi^2(z_i)\varphi^2{(g_1)}\varphi^3(z_i^{-1})\\
	&\ldots\\
	\varphi^{n-1}(g_i)&=\varphi^{n-1}(z_i)\varphi^{n-1}(g_1)\varphi^n(z_i^{-1}). 
	\end{split}
	\end{equation}
	In view of \lemref{normality}, for all $r=1, 2, \ldots, n$, we get $\overline{f}^r\circ d_h=d_{\widetilde{h}}\circ \overline{f}^r$ for some $\widetilde{h}\in \widehat{H'}$. Therefore $\varphi^r=d_{\widetilde{h}}\circ \overline{f}^r$ for all $r$.
	Since the field automorphism acts as an identity on the prime subfield $\Q$ of $k$, then $\overline{f}(g_i)=g_i$. Also, the diagonal automorphism $d_h$ acts by conjugation, where $h$ is an $(|\Phi|+|\Delta|)\times (|\Phi|+|\Delta|)$-diagonal matrix and $g_i$ is diagonal as well, so $d_h(g_i)=g_i$ for all $i$.
	Therefore $\varphi(g_i)=(\overline{f}\circ d_h)(g_i)=g_i$ for all $i$. 
	Therefore $\varphi^r(g_i)=g_i$ for all $r$.
	From equation \eqref{keyeq}, we get $g_i^n=z_ig_1^nd_{\widetilde{h}}(z_i^{-1})=z_ig_1^n\widetilde{h}z_i^{-1}\widetilde{h}^{-1}$, since $\overline{f}^n=\mathrm{Id}$.
Then \[g_i^n\widetilde{h}=z_i(g_1^n\widetilde{h})z_i^{-1},\] where $z_i\in G'_{\sigma}$ and $i=2, 3, \ldots$. Since $g_i^n\widetilde{h}=g(x_i)^n\widetilde{h}$ and $g(x_1)^n\widetilde{h}=g_1^n\widetilde{h}$ are conjugate in $\GL_{|\Phi|+|\Delta|}(k)$, their traces are equal for all $i=2, 3, \ldots$. In view of \lemref{lemma7}, we see that  $\mathrm{trace}(g(x_i)^n\widetilde{h})=\mathrm{trace}(g(x_1)^n\widetilde{h})\in k(T)\setminus k$, which contradicts \lemref{lemma6}. 
Therefore $R(\varphi)=\infty$.

This completes the proof.
\section{Isogredience classes and the $S_{\infty}$-property}\label{isogred}
Let $G$ be an arbitrary group. 
Suppose $\Psi\in \Out(G):=\Inn(G)\backslash\Aut(G)$. Two elements $\alpha, \beta\in \Psi$ are said to be \emph{isogredient} (or similar) if $\beta=i_h\circ \alpha \circ i_{h^{-1}}$ for some $h\in G$, where $i_h(g)=hgh^{-1}$ for all $g\in G$. Observe that  this is an equivalence relation on $\Psi$. Fix a representative $\gamma\in \Psi$. Then $\alpha=i_a\circ \gamma$ and $\beta=i_b\circ \gamma$ for some $a,b\in G$. Therefore $i_b\circ \gamma=\beta=i_h\circ \alpha \circ i_{h^{-1}}
=i_h\circ i_a\circ \gamma \circ i_{h^{-1}}$. Then we have  
\[i_b=i_h\circ i_a\circ (\gamma \circ i_{h^{-1}}\circ\gamma^{-1})=i_h\circ i_a\circ i_{\gamma(h^{-1})}=i_{ha\gamma(h^{-1})}.\]
Thus, $\alpha$ and $\beta$ are isogredient if and only if $b=ha\gamma(h^{-1})c$ for some $c\in Z(G)$ the center of $G$.
Let $S(\Psi)$ denote the number of isogredience classes of $\Psi$. If $\Psi=\overline{\textup{Id}}=\Inn(G)\textup{Id}_G$, then $S(\overline{\textup{Id}})$ is the number of usual conjugacy classes of $G/Z(G)$. 
A group $G$ has the \emph{$S_{\infty}$-property} if $S(\Psi)=\infty$ for all $\Psi\in \Out(G)$. Observe that if a group $G$ possesses the $S_{\infty}$-property then $G/Z(G)$ possesses the $R_{\infty}$-property, and hence by \lemref{msqtog} $G$ satisfies the $R_{\infty}$-property. The converse also holds if $Z(G)=\{e\}$. We record this as follows (cf. \cite[Theorem 3.4]{ft2015}):
\begin{lemma}\label{sinfty}
Let $G$ be an arbitrary group such that $Z(G)=\{e\}$. Then $G$ has the $R_{\infty}$-property if and only if $G$ has the $S_{\infty}$-property.
\end{lemma}

An automorphism $\varphi$ of $ G $ is said to be \textit{central} if  $ g^{-1}\varphi(g)\in Z(G) $ for all $ g \in G.$ Thus corresponding to every central automorphism $\varphi $ one can associate a homomorphism $ f_\varphi: G\to Z(G) $ given by 
$ f_\varphi(g)=g^{-1}\varphi(g).$ 
The following two lemmas appear in \cite[Propositions 1, 5, 12]{bnn2013}. We include a proof for the sake of completeness.
\begin{lemma}\label{lemma49}
	If the $\varphi$-conjugacy class $[e]_{\varphi}$ of the unit element $e$ of a group $G$ is a subgroup, then it is always a normal subgroup. 
	If $\varphi$ is a central automorphism of $G$,  then $[e]_{\varphi}$ is a subgroup of $G$. 
\end{lemma}
\begin{proof}
	Suppose	$[e]_{\varphi}=\{g\varphi(g)^{-1}\mid g\in G\}$ is a subgroup. Then for any  $ h\in G $, we get \[h(g\varphi(g)^{-1})h^{-1}=hg\varphi(g)^{-1}\varphi(h)^{-1}\varphi(h)h^{-1}=hg\varphi(hg)^{-1}(h \varphi(h)^{-1})^{-1}\in [e]_{\varphi}. \]
	For the second statement, assume that $ \varphi $ is a central automorphism. Then we can write $ \varphi(g)= gf_\varphi(g)$, where $ f_\varphi: G\to Z(G) $ is a homomorphism given by $f_{\varphi}(g)=g^{-1}\varphi(g)$. This implies 
	\[[e]_{\varphi}=\{g\varphi(g)^{-1}\mid g\in G\}=\{gf_\varphi(g^{-1})g^{-1}\mid g\in G\} =\{f_\varphi(g^{-1})\mid g\in G\}. \] Hence $ [e]_{\varphi} $ is a subgroup as $f_{\varphi}$ is a homomorphism.
\end{proof}
Let $N$ be a normal subgroup of $G$ which is stable under $\varphi$ (i.e., $\varphi(N)=N$), then we denoted by $\overline{\varphi}$, the automorphism induced by $\varphi$ on $G/N$ and by $\overline{e}$, the identity element of the factor group $G/N$. 
\begin{lemma}\label{lemma51}
	Let $G$ be a group such that for some automorphism $\varphi$ of $G$, the twisted conjugacy class $[e]_{\varphi}$ is a subgroup of $G$. Suppose that $N$ is a normal $\varphi$-stable subgroup of $G$. Then the $\overline{\varphi}$-conjugacy class $[\overline{e}]_{\overline{\varphi}}$ is a subgroup of $G/N$.
\end{lemma}
\begin{proof} Let $ x,y\in [\overline{e}]_{\overline{\varphi}}=\{g\varphi(g)^{-1}N\mid g\in G\}$. Then $ x= g\varphi(g)^{-1}N$ and $ y= h\varphi(h)^{-1}N$ for some $ g, h\in G $. This gives \[ xy=g\varphi(g)^{-1} h\varphi(h)^{-1}N=z\varphi(z)^{-1} N\in [\overline{e}]_{\overline{\varphi}},\] for some $ z\in G $ since 
	$[e]_{\varphi}$ is a subgroup of $G$. Similarly it can be shown that  
	$x^{-1}\in[\overline{e}]_{\overline{\varphi}}.$ Hence, the proof.
\end{proof}
The following result is similar to \lemref{msqtog} in the context of the $S_{\infty}$-property.
\begin{lemma}\cite[Lemma 2.3]{FN}\label{qtogs}
	Let $1\rightarrow N\rightarrow G\rightarrow Q\rightarrow 1$ be an exact sequence of groups. Suppose that $N$ is a characteristic subgroup of $G$ and $Q$ has the $S_{\infty}$-property, then $G$ also has the $S_{\infty}$-property.
\end{lemma}
\subsection{Proof of \corref{maincor1}}
Let $G'_{\sigma}$ be a twisted Chevalley group of adjoint type. Therefore $Z(G'_{\sigma})=\{e\}$. Now by \thmref{mainthm1} and \thmref{mainthm2} we know that $G'_{\sigma}$ has the $R_{\infty}$-property.
 Hence, in view of \lemref{sinfty}, the twisted Chevalley group $G'_{\sigma}$ (of adjoint type) has the $S_{\infty}$-property. Now let $G'_{\sigma}$ be any twisted Chevalley group, then we have the following short exact sequence (as in equation \eqref{qtog})
 \[1\rightarrow Z(G'_{\sigma})\rightarrow G'_{\sigma}\rightarrow G'_{\sigma}/Z(G'_{\sigma})\rightarrow 1,\]
 where $Z(G'_{\sigma})$ is a characteristic subgroup of $G'_{\sigma}$. Then by virtue of \lemref{qtogs}, $G'_{\sigma}$ has the $S_{\infty}$-property as $G'_{\sigma}/Z(G'_{\sigma})$ possesses the $S_{\infty}$-property.
 
This completes the proof.
\subsection{Proof of \corref{maincor2}}
Suppose that $[e]_{\varphi}$ is a subgroup of $G'_{\sigma}$, then in view of \lemref{lemma51}, $[\overline{e}]_{\overline{\varphi}}$ is a subgroup of $G'_{\sigma}/Z(G'_{\sigma})$. Therefore, by \lemref{lemma49}, $[\overline{e}]_{\overline{\varphi}}$ is a normal subgroup of $G'_{\sigma}/Z(G'_{\sigma})$. Once again recall that Steinberg proved that the twisted Chevalley group of adjoint type $G'_{\sigma}/Z(G'_{\sigma})$ is simple. Therefore either $[\overline{e}]_{\overline{\varphi}}=Z(G'_{\sigma})$ or $[\overline{e}]_{\overline{\varphi}}=G'_{\sigma}/Z(G'_{\sigma})$. Again by \thmref{mainthm1} and \thmref{mainthm2}, $G'_{\sigma}/Z(G'_{\sigma})$ possesses the $R_{\infty}$-property. Therefore $[\overline{e}]_{\overline{\varphi}}=Z(G'_{\sigma})$ which implies  $\overline{\varphi}=\mathrm{Id}_{G'_{\sigma}/Z(G'_{\sigma})}$. Hence $\varphi$ is a central automorphism of $G'_{\sigma}$. 

Conversely, suppose that $\varphi$ is a central automorphism of $G'_{\sigma}$. Then in view of \lemref{lemma49}, the $\varphi$-twisted conjugacy class $[e]_{\varphi}$ is a subgroup of $G'_{\sigma}$. 

This completes the proof.

\medskip
\textbf{Acknowledgement:} We would like to thank Timur Nasybullov for his  wonderful comments and suggestions on this work and also, for helping
with the proof of \lemref{normality}. We thank Swathi Krishna for correcting our English. We also would like to thank the anonymous referee for his/her careful reading and
for many helpful comments and suggestions which improved the readability of this paper.

\end{document}